\documentclass[12pt]{article}
\usepackage{amssymb}
\usepackage{mathrsfs}
\usepackage[hypertex]{hyperref}
\usepackage{amsfonts}
\usepackage{graphicx}
\usepackage{mathptmx}
\usepackage{latexsym,amsmath,amssymb,amsfonts,amsthm}
\usepackage{color}

\newcommand{\R}{\mathbb{R}}

\newtheorem{theorem}{Theorem}[section]
\newtheorem{lemma}[theorem]{Lemma}

\newtheorem{proposition}[theorem]{Proposition}
\newtheorem{remark}[theorem]{Remark}
\newtheorem{example}[theorem]{Example}

\def\vd{\mathrm{d}}

\def\bthm{\begin{theorem}}
\def\ethm{\end{theorem}}
\def\blem{\begin{lemma}}
\def\elem{\end{lemma}}
\def\brem{\begin{remark}}
\def\erem{\end{remark}}
\def\bexm{\begin{example}}
\def\eexm{\end{example}}

\def\ecor{\end{corollary}}
\def\r{\right}
\def\l{\left}

\def\be{\begin{equation}}
\def\de{\end{equation}}

\def\F{\mathscr{F}}

\def\R{\mathbb{R}}
\def\E{\mathbb{E}}
\def\M{\mathscr{M}}
\def\G{\mathcal{G}}
\def\P{\mathbb{P}}
\def\O{\mathscr{O}}
\def\Ric{\text{Ric}}
\def\|{/\!/}

\def\proclaim#1{\bigskip\noindent{\bf #1}\bgroup\it\  }
\def\endproclaim{\egroup\par\bigskip}

\begin{document}
\setcounter{page}{1}
\title{A probabilistic method for gradient estimates of some geometric flows}
\author{Xin Chen$^{\dag}$, Li-Juan Cheng$^{\ddag}$ and Jing Mao$^{\sharp,\ast}$}
\date{}
\protect\footnotetext{\!\!\!\!\!\!\!\!\!\!\!\!{ $^\ast$Corresponding
author} 
}
\maketitle ~~~\\[-15mm]
\begin{center}
{\footnotesize  $^{\dag}$ Department of Mathematics, Shanghai Jiao
Tong University, Shanghai, 200240, P.R. China.
 \\
Email: chenxin\_217@hotmail.com}

 {\footnotesize  $^{\ddag}$ Department of Applied Mathematics, Zhejiang University of Technology, Hangzhou, 310023, China \\
Email: chenglj@zjut.edu.cn} \\
 {\footnotesize  $^{\sharp}$Department of Mathematics, Harbin Institute of Technology
(Weihai), Weihai, 264209, China \\
Email: jiner120@163.com, jiner120@tom.com}
\end{center}


\begin{abstract}
In general, gradient estimates are very important and necessary
for deriving convergence results in different geometric flows, and
most of them are obtained by analytic methods. In this paper, we
will apply a stochastic approach to  systematically give
 gradient estimates for some important geometric quantities under
the Ricci flow, the mean curvature flow, the forced mean curvature
flow and the Yamabe flow respectively. Our conclusion gives another
example that probabilistic tools can be used to simplify proofs for
some problems in geometric analysis.
\end{abstract}

\markright{\sl\hfill  X. Chen, L. J. Cheng and J. Mao\hfill}

\section{Introduction}
\renewcommand{\thesection}{\arabic{section}}
\renewcommand{\theequation}{\thesection.\arabic{equation}}
\setcounter{equation}{0} \setcounter{maintheorem}{0}

The study on geometric flows is an important and hot topic in
geometric analysis within the past three decades. Since R. S.
Hamilton introduced the so-called Ricci-Hamilton flow in \cite{Ham}
and used it to deform the Riemannian metric on a $3$-dimensional
compact manifold with positive Ricci curvature to get the metric of constant positive
curvature, the most marvellous application of the Ricci flow so far
is that G. Perelman \cite{p1} used the Ricci flow to give an amazing
proof for the $3$-dimensional Poincar\'{e} Conjecture. At the same time, many
mathematicians tried to consider other curvature flows, which also
have a wide range of applications in geometric analysis, for
instance, the (smooth) mean curvature flow introduced by G. Huisken
\cite{Hus}.

Such geometric flows are in essential considering given initial
manifolds (or submanifolds) evolving under prescribed geometric
evolution equations, which often correspond to systems of
second-order parabolic or hyperbolic partial differential equations.
Usually, gradient estimates for geometric quantities (such as
the Riemannian curvature, the second fundamental form, etc) on the
evolving manifolds have important applications in many problems
concerning the geometric flows, e.g. ``the existence of a solution
to the corresponding evolution equation on some non-compact
manifold'', ``the convergence property of the evolving manifolds'',
``whether singularities appear or not during the evolving process,
and then if they appear, how to analyse their properties?'', and so
on. For the case of the Ricci flow, we refer  readers to
\cite[Chapters 6-8]{BLN}, \cite{Shi}, \cite{Shi1} for detailed
introduction about the application of  (local) gradient estimates
in such problems.

Here we also want to give some remarks about the stochastic methods
on  gradient estimates for various differential equations or systems
on Riemannian manifolds. The coupling methods were first applied by
M. Cranston \cite{CR} to   gradient estimates for the solutions to
elliptic equations on a Riemannian manifold. Furthermore, we refer
readers to \cite[Chanpter 2]{Wbook1}, \cite{Wbook2} and references
therein for  further development of such coupling methods. A.
Thalmaier and F.-Y. Wang \cite{TW} (see also \cite{ADT}) used the
martingale methods to give a stochastic proof for local gradient
estimate of the harmonic function on a regular domain. Such
martingale methods were also used to tackle gradient estimates for
various different models, e.g. the harmonic maps between manifolds
investigated by M. Arnaudon and A. Thalmaier \cite{AT}, the tensor
bundles associated with a (linear) heat equation studied by B. K.
Driver and A. Thalmaier \cite{Driver}, and the scalar curvature
under a two dimensional Ricci flow investigated by K. A. Coulibaly
\cite{KC} (through using the Brownian motion associated with a
time-changing metric, see also \cite{ACT}).

As observed in the references above, the stochastic methods can be
applied to  gradient estimates for the linear parabolic or
elliptic equations, while most of gradient estimates for
(non-linear) geometric flows are obtained by analytic methods. In
this article, we will generalize the methods in \cite{Driver} and
\cite{TW} to investigate   gradient estimates for many geometric
flows. In particular, such methods will be systematically used to
give   gradient estimates for some geometric quantities of a large
class of geometric flows, e.g. the Riemannian curvature of the Ricci
flow, the second fundamental form of the mean curvature flow and
forced mean curvature flows, the scalar curvature of the Yamabe
flow, and so on.

The paper is organized as follows. In Section 2, we provide a brief
introduction of the Brownian motion on a manifold associated with a
time-changing metric introduced in \cite{ACT,KC}, and then present
 an It\^{o} formula for  tensor fields over a  Brownian motion with
respect to some time-changing metric (see Proposition \ref{p1} for
the precise statement), which will be frequently used in this
article. In Section 3, inspired by \cite{TW}, we first show local
martingale arguments for the quantities determined by the equation
(\ref{e3-1}) (see Theorem \ref{T1}), and then based on this result,
we give a general derivative formula for smooth sections of tensor
bundles $T^{m,k}M$ satisfying specific equations. In the last
section, as applications of Theorem \ref{c1}, we can always give
local gradient estimates for the Riemannian curvature tensor of the
Ricci flow (see Theorem \ref{T4-1}), the second fundamental form for
the mean curvature flow (see Theorem \ref{T4-2}) and two kinds of
forced mean curvature flows (see Theorems \ref{theoremxx1} and
\ref{theoremxx2}), and the scalar curvature of the Yamabe flow (see
Theorem \ref{theoremxx3} for the details).

\section{Brownian motion associated with the time changing metric}
\renewcommand{\thesection}{\arabic{section}}
\renewcommand{\theequation}{\thesection.\arabic{equation}}
\setcounter{equation}{0} \setcounter{maintheorem}{0}

 Let $M$ be an $n$-dimensional differential manifold equipped with a time-changing
Riemannian metric $\{g_t\}_{t\in [0,T_c)}$,\ $T_c\in (0,\infty]$,
such that $(M,g_t)$ is  complete for every $t \in [0,T_c)$. Let
$\nabla^t$ be the $g_t$-Levi-Civita connection on tensor fields. We
denote the frame bundle over $M$ by $\F(M)$ with $\pi: \F(M)
\rightarrow M$ being the associated canonical projection from
$\F(M)$ to $M$. We fix the standard orthonormal basis
$\{e_i\}_{i=1}^n$ of $\R^n$, i.e., $e_i=(0,\cdots,
\underbrace{1}_{i\ {\rm th}}, \cdots, 0)$, $1\le i \le n$.

For each fixed $t \in [0,T_c)$, we define a smooth vector field
$H_i^t$ on $\F(M)$ such that for every $u \in \F(M)$,
$H_i^t(u):=\mathbf{H}^t(u e_i)$, where $\mathbf{H}^t$ is the
horizontal lift operator with respect to $\nabla^t$. We denote the
set of $n \times n$ non-degenerate matrices by $GL_n$, and the set
of $n \times n$ matrices by $\M$. For each $u \in \F(M)$, let  $l_u:
GL_n \rightarrow \F(M)$  be the left multiplication operator defined
as
\begin{eqnarray} \label{e1}
\pi(l_u(A))=\pi(u),\ \qquad \ l_u(A)e_i=u(Ae_i),\ \ \ \mbox{for
all}\ \ A \in GL_n,\ 1\le i \le n.
\end{eqnarray}
Let $\{E_{ij}\}_{i,j=1}^n$ be the canonical basis of $\M$ and
$V_{ij}(u):=Tl_u(E_{ij})$, where $Tl_u: \M \rightarrow T_u \F(M)$ is
the differential of $l_u$.

Let $T_x^*M:=(T_x M)^*$ be the dual space of $T_x M$, and for
non-negative integers $m,k$, the bundle of $(m,k)$ tensors
$T^{m,k}M$ is defined as follows
\begin{equation*}
T^{m,k}M:=\bigcup_{x \in M}(T_x M)^{\otimes m}\otimes (T_x^*
M)^{\otimes k},
\end{equation*}
where $\otimes$ is the tensor product operator. We denote by
$\Gamma(T^{m,k}M)$ the set of smooth sections of $T^{m,k}M$, and for
every pair of tensor bundles $E$ and $\tilde E$, denote by
$\Gamma({\rm Hom}(E;\tilde E))$ the collection of smooth
sections of linear operator from $E$ to $\tilde E$. Clearly, we can
identify $\Gamma({\rm Hom}(E;\tilde E))$ with $\Gamma\big(E^*
\otimes \tilde E\big)$.

Now we can introduce the Brownian motion on $M$ with respect to the
time-changing metric $\{g_t\}_{t \in [0,T_c)}$ which is first given
in \cite{ACT,KC}.
Let $(\mathscr{O},g_t)$ be the orthonormal frame bundle with respect
to the metric $g_t$.
 Now, consider the following Stratonovich stochastic differential equation (write SDE for short)
\begin{equation}\label{e2}
\begin{cases}
& dU_t=\sum_{i=1}^n \sqrt{2}H_i^t(U_t)\circ \vd W_t^i-\frac{1}{2}\sum_{i,j=1}^n G_{ij}(t,U_t)V_{ij}(U_t)\vd t,  \qquad t\in[0,\zeta),\\
& U_t|_{t=0}=u_0
\end{cases}
\end{equation}
 on $M$, where
$W_t$ is a $\R^n$-valued Brownian motion on a probability space
$(\Omega, \F, \P)$ with corresponding filtration $\F_t$,
 $u_0 \in (\O,g_0)$ is non-random with $\pi(u_0)=x_0$,
$G_{ij}(t,u):$ $=\partial_t g_t(\pi(u))\big(ue_i,ue_j\big)$ for
every $u \in \F(M)$, and $\zeta \le T_c$ is the maximal time for the
solution of (\ref{e2}). Clearly, $\zeta=T_c$ if $M$ is compact, and
moreover, by \cite[Theorem 1]{KP}, if $g_t$ evolves as the super
backwards Ricci flow (i.e., $\partial_t g_t \le \Ric_t$) and
$(M,g_t)$ is complete for every $t \in [0,T_c)$, then $\zeta=T_c$.
Besides, by \cite[Proposition 1.2]{KC} (see also \cite{ACT}), we
have $U_t \in (\O,g_t)$ for every $t \in [0, \zeta)$, and $\langle
u_0e_i, u_0e_j \rangle_{g_0}= $ $\langle U_te_i, U_te_j
\rangle_{g_t}$. Then we define $X_t:=\pi(U_t)$ and $U_t$ as the
Brownian motion and the stochastic horizontal lift associated with
the time-changing metric $\{g_t\}$, respectively. Let
$/\!/_{s,t}:=U_t U_s^{-1}: T_{X_s}M$ $\rightarrow T_{X_t}M$, $0 \le
s \le t < \zeta$ be the stochastic parallel translation  along
$X_{\cdot}$. Clearly, $/\!/_{s,t}$ is an isometry from $(T_{X_s}M,
g_s)$ to $(T_{X_t}M, g_t)$.

We can define the scalarization of a tensor bundle as in \cite{Hsu}.
For each $u \in \F(M)$ with $\pi(u)=x$, let $Y_i=ue_i$, $1 \le i \le
n$. Then $\{Y_i\}$ is a basis of $T_x M$ with the dual basis
$\{Y^i\}\subseteq T^*_x M$, and for every $\theta \in T^{m,k}_x M$,
it has a unique expression under $\{Y_i\}, \{Y^i\}$ as follows
\begin{equation*}
\theta=\theta_{j_1\dots j_k}^{i_1\dots i_m}(u)Y_{i_1} \otimes \dots
\otimes Y_{i_m}\otimes Y^{j_1}\otimes \dots \otimes Y^{j_k},
\end{equation*}
for some $\theta_{j_1\dots j_k}^{i_1\dots i_m}(u)\in \R$ (which
depends on $u$), where the upper and lower indices mean the
summation of different items. Now, we can define the scalarization
$\tilde \theta$ at $u \in \F(M)$ by
\begin{equation*}
\tilde \theta(u):=\theta_{j_1\dots j_k}^{i_1\dots i_m}(u)e_{i_1}
\otimes \dots \otimes e_{i_m}\otimes e^{j_1}\otimes \dots \otimes
e^{j_k},
\end{equation*}
where $\{e^i\}$ denotes the dual basis of $\{e_i\}$ in $\R^n$.

  We define the stochastic parallel translation (on tensors)
$/\!/_{0,t}: T_{x_0}^{m,k}M \rightarrow T_{X_t}^{m,k}M$,
$t\in[0,\zeta)$ along $X_{\cdot}$ as follows, for every (random)
$\vartheta \in T_{X_t}^{m,k}M$, $ /\!/_{0,t}^{-1}\vartheta:=\tilde
\vartheta(U_t)$, where $\tilde \vartheta(U_t)$ is the scalarization
of $\vartheta$ at $U_t$. In particular, for a smooth section $\theta
\in \Gamma(T^{m,k}M)$,
\begin{eqnarray*}
/\!/_{0,t}^{-1} \theta(\pi(U_t))=\theta_{j_1\dots j_k}^{i_1\dots
i_m}(U_t)e_{i_1} \otimes \dots \otimes e_{i_m}\otimes e^{j_1}\otimes
\dots \otimes e^{j_k}.
\end{eqnarray*}
 Clearly, here we identify $(T_{x_0}M, g_0)$ with $\R^n$
by the isometry $U_0=u_0$. \emph{Now, we make an agreement that
throughout this paper, for convenience, we write $/\!/_{0,t}$ as
$/\!/_{t}$ and identify $(T_{x_0}^{m,k} M, g_0)$ with $T^{m,k}\R^n$
by the isometry $U_0=u_0$ once the initial point $u_0$ of the
stochastic horizontal lift is specified}. So we use $|\cdot|$ and
$\langle \cdot,\cdot \rangle$ to denote the norm and metric
respectively both on $T^{m,k}\mathbb{R}^n$ and
$(T^{m,k}_{x_0}M,g_0)$.

For every $\theta \in \Gamma(T^{m,k}M)$, $t \in [0,T_c)$, we are
able to define a linear operator $\G^t \in \Gamma({\rm
Hom}(T^{m,k}M;$ $T^{m,k}\R^n))$ as follows, for each $u\in
\mathscr{F}(M)$,
\begin{equation}\label{e1-1}
\begin{split}
\G^t(u)\theta(\pi u)&:=\sum_{p,q=1}^n
\frac{G_{pq}(t,u)}{2}\sum_{l=1}^k\sum_{j_l=p}\theta_{j_1\dots
j_k}^{i_1\dots i_m}(u)e_{i_1} \otimes \dots \otimes e_{i_{m}}
\otimes e^{j_1} \otimes \dots \otimes \underbrace{e^{q}}_{l\ {\rm th}} \dots \otimes e^{j_k}\\
&-\sum_{p,q=1}^n
\frac{G_{pq}(t,u)}{2}\sum_{l=1}^m\sum_{i_l=q}\theta_{j_1\dots
j_k}^{i_1\dots i_m}(u)e_{i_1} \otimes \dots \otimes
\underbrace{e_{p}}_{l\ {\rm th}} \dots\otimes e_{i_{m}} \otimes
e^{j_1} \otimes \dots \otimes e^{j_k}.
\end{split}
\end{equation}

By applying a similar method to \cite[Theorem 3.2]{KC}, we can show
the following It\^{o} formula.
\begin{proposition}\label{p1}
For every $\theta \in \Gamma(T^{m,k}M)$ and stopping time
$\tau<\zeta$, we have
\begin{equation}\label{p1.1}
\begin{split}
& \vd \big(\|_t^{-1}\theta(X_t)\big)=\sum_{i=1}^n\sqrt{2}
\|_t^{-1}\nabla^t_{U_t \vd  W_t^i}\theta(X_t)
+\Big(\|_t^{-1}\Delta^t\theta(X_t)-\G^t(U_t)\theta(X_t)\Big)\vd t, \
\quad 0\le t <\tau,
\end{split}
\end{equation}
where $\Delta^t: \Gamma(T^{m,k}M) \rightarrow \Gamma(T^{m,k}M)$ is
the connection Laplacian operator on $\Gamma(T^{m,k}M)$ associated
with the metric $g_t$, i.e., for any $x\in M$ and $u\in (\O, g_t)$
with $\pi u=x$,
$\Delta^t\theta(x):=\sum_{i=1}^n(\nabla^t)^2\theta(ue_i,ue_i)$.
\end{proposition}
\begin{proof}
From the definition explained above, we know that
$\|_t^{-1}\theta(X_t)=\tilde \theta(U_t)$. Then applying the It\^{o}
formula to $\tilde \theta(U_t)$,  we can obtain that for every $0\le
t <\tau$,
\begin{equation}\label{p1.2}
\begin{split}
& \vd \big(\|_t^{-1}\theta(X_t)\big)=\sum_{i=1}^n \sqrt{2} H_i^t
\tilde \theta(U_t)\vd W_t^i +\Big(\sum_{i=1}^n H_i^tH_i^t \tilde
\theta(U_t)-\frac{1}{2} \sum_{i,j=1}^n G_{ij}(t,U_t)V_{ij}\tilde
\theta(U_t) \Big)\vd t.
\end{split}
\end{equation}
Since $H_i^t(U_t)=\mathbf{H}^t(U_te_i)$ and $U_t \in (\O, g_t)$, by
\cite[Proposition 2.2.1]{Hsu}, we have
\begin{equation} \label{add-1}
H_i^t \tilde \theta(U_t)=(\widetilde{\nabla^t_{U_te_i}\theta})(U_t)
=\|_t^{-1}\nabla^t_{U_te_i}\theta(X_t)
\end{equation}
and from \cite[Page 193]{Hsu}, we know that
\begin{equation} \label{add-2}
\sum_{i=1}^n H_i^t H_i^t \tilde \theta^t(U_t)=(\widetilde{\Delta^t
\theta})(U_t) =\|_t^{-1}\Delta^t \theta(X_t).
\end{equation}
On the other hand, note that $V_{pq}(U_t):=Tl_{U_t}(E_{pq})$, so we
have
\begin{align}\label{add-3}
V_{pq}\tilde \theta(U_t)&=\frac{\vd}{\vd\varepsilon}\big|_{\varepsilon=0}\tilde \theta(U_tA(\varepsilon))\nonumber\\
&=\sum_{l=1}^k\sum_{j_l=p}\theta_{j_1\dots j_k}^{i_1\dots
i_m}(U_t)e_{i_1} \otimes \dots \otimes e_{i_{m}}
\otimes e^{j_1} \otimes \dots \otimes \underbrace{e^{q}}_{l\ {\rm th}} \dots \otimes e^{j_k}\nonumber\\
&-\sum_{l=1}^m\sum_{i_l=q}\theta_{j_1\dots j_k}^{i_1\dots
i_m}(U_t)e_{i_1} \otimes \dots \otimes \underbrace{e_{p}}_{l\ {\rm
th}} \dots\otimes e_{i_{m}} \otimes e^{j_1} \otimes \dots \otimes
e^{j_k},
\end{align}
where $A:[-1,1]\rightarrow GL_{n}$ satisfies $A(0)=\mathbf{I}$,
$\frac{\vd}{\vd\varepsilon}\big|_{\varepsilon=0}A(\varepsilon)=
E_{ij}$, with $\mathbf{I}$ being the identity $n \times n$ matrix.

Therefore, putting \eqref{add-1}, \eqref{add-2} and \eqref{add-3}
into  \eqref{p1.2},  we can show (\ref{p1.1}).
\end{proof}

\section{Derivative Formula}
\renewcommand{\thesection}{\arabic{section}}
\renewcommand{\theequation}{\thesection.\arabic{equation}}
\setcounter{equation}{0} \setcounter{maintheorem}{0}

In this section, we will consider general derivative formula for
a large class of tensor fields satisfying some (non-linear)
parabolic equations. From now on, we assume that $a_t \in
\Gamma(T^{m,k}M)$, $t \in [0,T_c)$ for some non-negative integers
$m,k$. Moreover, for a time changing metric $\{g_t\}_{0\le t <T_c}$,
$a_t$ and $\nabla^t a_t \in \Gamma(T^{m,k+1}M)$ satisfy the
following equations
\begin{equation}\label{e3-1}
\begin{split}
& \frac{\partial}{\partial t}a_{t}=-\Delta^t a_{t}+ F_t a_t,\\
& \frac{\partial}{\partial t}\nabla^t a_t=-\Delta^t(\nabla^t a_t)+
\hat F_t \nabla^t a_t,
\end{split}
\end{equation}
where $F_t \in \Gamma({\rm Hom}(T^{m,k}M;T^{m,k}M))$, and $\hat F_t
\in \Gamma({\rm Hom}(T^{m,k+1}M;T^{m,k+1}M))$.

As before, in this section, let $X_t$, $U_t$ be the Brownian motion
and the stochastic horizontal lift corresponding to the
time-changing metric $\{g_t\}_{0\leq t<T_c}$ with the initial point
$X_0=x_0$, $U_0=u_0$ and the maximal time $\zeta\le T_c$ respectively. Let $\{Q_{t}\}_{0\leq
t<\zeta}\in {\rm Hom}(T^{m,k}_{x_0}M;T^{m,k}_{x_0}M)$ and
$\{\hat{Q}_{t}\}_{0\leq t<\zeta}\in {\rm
Hom}(T^{m,k+1}_{x_0}M;T^{m,k+1}_{x_0}M)$ be the solutions to the
following ordinary differential equations
\begin{align}\label{Q1}
\frac{\vd }{\vd t}Q_{t}=-Q_{t}\big(
(F_t)_{\|_t}-\mathcal{G}^t_{\|_t}\big) \ \ \ \mbox{and}\ \
Q_{0}={\rm id},
\end{align}
and
\begin{align}\label{Q2}
\frac{\vd }{\vd t}\hat{Q}_{t}=-\hat{Q}_{t}\big((\hat
F_t)_{\|_t}-\mathcal{G}^t_{\|_t} \big) \ \ \ \mbox{and}\ \
\hat{Q}_{0}={\rm id},
\end{align}
respectively, where $\mathcal{G}^t_{\|_t}, (F_t)_{\|_t} \in {\rm
Hom}$ $(T^{m,k}_{x_0}M;T^{m,k}_{x_0}M)$ such that
\begin{equation*}
\mathcal{G}^t_{\|_t}:=\mathcal{G}^t\|_{t}, \ \
(F_t)_{\|_t}:=\|_t^{-1}F_t\|_t
\end{equation*}
with $\mathcal{G}^t$ defined by (\ref{e1-1}), and all the items in
(\ref{Q2}) similarly defined.

Inspired by \cite[Proposition 3.2]{Driver}, we can derive the
following theorem on a local martingale argument, which will be
applied to gradient estimates for some geometric flows in the
next section.
\begin{theorem}\label{T1}
Suppose  that $\{a_{t}\}_{0\leq t< T_c}$ satisfies  the equation
\eqref{e3-1}. For $t\in [0, \zeta)$, we define
\begin{equation}\label{T1-1}
 N_{t}:=Q_t\|_t^{-1}a_{t}(X_t),\ \ \   \hat{N}_t:=\hat{Q}_{t}\|_t^{-1}\nabla^t a_{t}(X_t).
\end{equation}
Suppose that $l_t$ and $\hat l_t$ are some $T_{x_0}^{m,k}M$ and
$T^{m,k+1}_{x_0}M$ valued continuous semi-martingales respectively,
which have the following forms
\begin{equation}\label{T1-2}
\begin{split}
& \vd l_t=\sum_{i=1}^n\sqrt{2}\alpha_t^i \vd W_t^i,\\
& \vd\hat l_t=
 \hat \beta_t
\vd t.
\end{split}
\end{equation}
Here $\alpha^i_{\cdot}$, $1\le i \le n$ are
adapted $T^{m,k}_{x_0}M$ valued processes,
$\hat \beta_{\cdot}$ is an adapted
$T^{m,k+1}_{x_0}M$ valued processes.
We also assume
 \begin{equation}\label{T1-3}
\begin{split}
& 2 Q_t^{\mathrm{tr}} \alpha_t
=\hat Q_t^{\mathrm{tr}} \hat \beta_t,
\end{split}
\end{equation}
where we view $\alpha_t:=(\alpha_t^1,\dots,\alpha_t^n)$ with
$\alpha_t(u_0e_i)=\alpha_t^i$ as an adapted $T^{m,k+1}_{x_0}M$
valued process, $Q_t^{\mathrm{tr}}
\alpha_t:=(Q_t^{\mathrm{tr}}\alpha_t^1,\dots,
Q_t^{\mathrm{tr}}\alpha_t^n)$, and $Q_t^{\mathrm{tr}}$, $\hat
Q_t^{\mathrm{tr}}$ denote the transposes of $Q_t$ and $\hat Q_t$
respectively.  Then $Z_t:=\big\langle \hat N_t, \hat l_t\big\rangle$
$-\big\langle N_t, l_t\big\rangle$ ($t \in [0,\zeta)$) is a local martingale.
\end{theorem}
\begin{proof}
By the It\^{o} formula in Proposition \ref{p1} and  equations
(\ref{e3-1}), (\ref{Q1}),
\begin{equation}\label{T1-4}
\begin{split}
\vd N_t&=\sqrt{2}Q_t\|_t^{-1}\nabla^t_{U_t \vd W_t}a_t(X_t)+
Q_t\big(\|_t^{-1}\Delta^t a_t(X_t)-\G^t(U_t)a_t(X_t)\big)\vd t\\
&-Q_t\|_t^{-1}\big(\Delta^t a_t(X_t)-F_t(X_t)a_t(X_t)\big)\vd t
-Q_t\big(\|_t^{-1}F_t(X_t)a_t(X_t)-\G^t(U_t)a_t(X_t)\big)\vd t\\
&=\sqrt{2}Q_t\|_t^{-1}\nabla^t_{U_t \vd W_t}a_t(X_t).
\end{split}
\end{equation}
Similarly, we have
\begin{equation}\label{T1-5}
\vd \hat{N}_t=\sqrt{2} \hat Q_t\|_t^{-1}\nabla^t_{U_t \vd
W_t}\big(\nabla^t a_t(X_t)\big).
\end{equation}
We write $\vd X_t\simeq \vd Y_t$ if $X_t-Y_t$ is a local martingale.
Then from (\ref{T1-4}) and (\ref{T1-5}), we can obtain
\begin{equation*}
\begin{split}
\vd Z_t & =\langle \vd \hat N_t, \hat l_t\rangle+\langle \hat N_t,
\vd \hat l_t\rangle-
\langle \vd N_t, l_t\rangle- \langle N_t, \vd l_t\rangle-\langle \vd N_t, \vd l_t\rangle\\
&\simeq \langle \hat{Q}_{t}\|_t^{-1}\nabla^t a_{t}(X_t), \hat
\beta_t \rangle \vd t-2\sum_{i=1}^n
\langle Q_t\|_t^{-1}\nabla^t_{U_t e_i}a_t(X_t), \alpha_t^i\rangle \vd t\\
&=\langle \|_t^{-1}\nabla^t a_{t}(X_t), \hat Q_t^{\mathrm{tr}}\hat
\beta_t \rangle \vd t- \langle \|_t^{-1}\nabla^t a_t(X_t),
2Q_t^{\mathrm{tr}}\alpha_t\rangle \vd t.
\end{split}
\end{equation*}
Therefore, substituting  the condition (\ref{T1-3}) into the above
equality, we can show immediately that $Z_t$ is a local martingale.
\end{proof}

Based on Theorem \ref{T1}, we get the following result.
\begin{theorem}\label{c1}
Given a fixed $T<T_c$ and an open domain $U \subseteq M$ such that
$x_0 \in U$, let $\tau$ be a stopping time such that $X_t \in U$ as
long as $t<\tau$. Suppose that $h_{\cdot}$ is an adapted
$T^{m,k+1}_{x_0}M$ valued process such that $h_0=v$ for some
non-random $v \in T^{m,k+1}_{x_0}M$, $h_{T\wedge\tau}=0$, and $t
\rightarrow h_t$ is absolutely continuous with
$\int_0^{T}|\dot{h}_t|^2 \vd t<\infty$, where $\dot{h}_t:=\frac{\vd
h_t}{\vd t}$. Suppose that $a_t$ satisfies the equation \eqref{e3-1}
on $[0,T]\times U$, and we define a $T^{m,k}_{x_0}M$ valued process
$l^h$ as follows
\begin{equation}\label{c1-1}
l^h_t:=\frac{\sqrt{2}}{2}\int_0^t (Q_s^{\mathrm{tr}})^{-1}\hat
Q_s^{\mathrm{tr}}\dot{h}_s \vd W_s,\ \ 0\le t \le T \wedge \tau,
\end{equation}
where $Q_t$, $\hat Q_t$ are defined by \eqref{Q1} and \eqref{Q2}
respectively. Moreover, if we assume
\begin{equation}\label{c1-2}
\begin{split}
&\sup_{t \in [0,T]}\E\Big[\big|\langle \hat N_{t\wedge \tau},h_{t
\wedge \tau} \rangle\big|^{1+\delta}\Big]<\infty, \ \ \quad \sup_{t
\in [0,T]}\E\Big[\big|\langle N_{t \wedge \tau}, l^h_{t \wedge
\tau}\rangle\big|^{1+\delta}\Big]<\infty
\end{split}
\end{equation}
for some $\delta>0$, where $N_t$, $\hat N_t$ are defined by
\eqref{T1-1}, then we have
\begin{equation}\label{c1-3}
\langle \nabla^0 a_0(x), v \rangle=-\E\Big[\langle N_{T \wedge
\tau}, l_{T \wedge \tau}^h\rangle\Big].
\end{equation}
\end{theorem}
\begin{proof}
Let $Z_t:=\langle \hat N_t, h_t\rangle-\langle N_t, l_t^h\rangle$, $t \in [0, T \wedge \tau]$.
Note that (\ref{c1-1}) implies that the condition (\ref{T1-3}) in
Theorem \ref{T1} is true for the processes $h$ and $l^h$.  Although
we only assume $a_t$ satisfies the equation (\ref{e3-1}) on
$[0,T]\times U$, following the same steps in the proof of Theorem
\ref{T1}, we still know that $Z_{t\wedge \tau}$ is a local
martingale since $X_t \in U$ as long as $t<\tau$. Note that the
uniformly integrable condition (\ref{c1-2}) ensures that $Z_{t\wedge
\tau}$ is a real martingale, we have $\E \big[Z_0]=\E \big[Z_{T
\wedge \tau}\big]$, which is (\ref{c1-3}).
\end{proof}

\section{Application to geometric flows}
\renewcommand{\thesection}{\arabic{section}}
\renewcommand{\theequation}{\thesection.\arabic{equation}}
\setcounter{equation}{0} \setcounter{maintheorem}{0}

 In this section, we will apply Theorem \ref{c1} to get some (local)
gradient estimates for various geometric flows.
\subsection{Local gradient estimate for the  Riemannian curvature tensor of the Ricci flow}
We consider the following Ricci flow about the metric $g_t$ on an
$n$-dimensional manifold $M$ ($n \ge 2$)
\begin{equation}\label{e4-1}
\begin{cases}
&\frac{\partial }{\partial t}g_t(x)=-2\text{Ric}_t(x),\ \ 0 \le t <T_c, \\
&g_0(x)=\mathring{g}(x),
\end{cases}
\end{equation}
where $\text{Ric}_t$ is the Ricci curvature tensor associated with
the metric $g_t$, and $\mathring{g}$ is the initial metric for the
flow.

The Ricci flow  was firstly introduced in \cite{Ham} by R. S.
Hamilton, where the local existence of (\ref{e4-1}) was proved when
$M$ is compact. We also refer readers to \cite{BLN,Top} for an
overall introduction of the Ricci flow. If $M$ is non-compact and
the initial metric $\mathring{g}$ is complete with bounded sectional
Riemannian curvature, the local existence of (\ref{e4-1}) was
established in \cite{Shi}. In this subsection, we will apply Theorem
\ref{c1} to give a stochastic proof for the local gradient estimate
obtained in \cite{Shi}.

Throughout this subsection, we assume that the equation (\ref{e4-1})
holds on $[0,T]\times U$ for some constant $0<T<T_c$ and open domain
$U \subseteq M$, and $(M, g_t)$ is complete for every $t \in
[0,T_c)$, but with no requirement that $M$ is compact.  Then by
\cite[Theorem 13.2]{Ham} (see also \cite[Page 49 (3.3.2)]{Top} and
\cite[Page 50 (3.3.3)]{Top}), we know that on $[0,T]\times U$, the
evolution equations
\begin{equation}\label{e4-2}
\begin{split}
& \frac{\partial }{\partial t}{\rm Rm}_t=\Delta^t {\rm Rm}_t+ {\rm Rm}_t*{\rm Rm}_t,\\
& \frac{\partial }{\partial t}\nabla^t{\rm Rm}_t= \Delta^t \nabla^t
{\rm Rm}_t+ {\rm Rm}_t*\nabla^t{\rm Rm}_t
\end{split}
\end{equation}
hold, where ${\rm Rm}_t$ denotes the Riemannian curvature tensor
with respect to $g_t$, and $*$ is a contraction operator on tensors
such that $|A*B|_t \le C(n) |A|_t |B|_t$ for every tensor $A, B$
with $t \in [0,T]$, where $|\cdot|_t$ denotes the norm for
tensors induced by the metric $g_t$, and $C(n)$ is some constant
independent of $A$, $B$, $t$ (see also \cite[Page 26]{Top}).

In this subsection, we define $a_t:={\rm Rm}_{T-t}$ for every $t \in
[0,T]$. Then $a_t \in \Gamma(T^{1,3}M)$. By (\ref{e4-2}) we can
easily check that $a_t$ satisfies the equation (\ref{e3-1}) with
respect to the time changing metric $\{\tilde g_t:=g_{T-t}\}_{t \in
[0,T]}$, and the maps $F_t$, $\hat F_t$ satisfy $F_t \theta=-{\rm
Rm}_{T-t}*\theta$, $\hat F_t \hat \theta=-{\rm Rm}_{T-t}*\hat
\theta$ for every $\theta \in \Gamma(T^{1,3}M)$ and $\hat \theta \in
\Gamma(T^{1,4}M)$ with $*$ being the contraction operator in the
first equation and the second equation of (\ref{e4-2}) respectively.

Inspired by the methods in \cite[Theorem 4.1]{TW} and
\cite[Corollary 5.1]{TW} (see also \cite{cheng}), together with
Theorem \ref{c1},  we can construct a suitable process $h$ to get
the following local gradient estimate which has been shown by W. X.
Shi \cite{Shi} via some analytic method.

\emph {From now on, the constant $C$ in the conclusions of
theorems will be indexed by all the parameters which it depends on,
and the constant $C$ in the proofs of theorems may change in
different lines but only depends on $n$.}

\begin{theorem}[\cite{Shi}]\label{T4-1}
Suppose that for some constants $K>0$, $K_1>0$, $|{\rm Rm}_t(x)|_t
\le K$, $|{\rm Ric}_t(x)|_t\le K_1$ for every $(t,x) \in [0,T]\times
U$, and there exist $x_0 \in M$, $r>0$ such that
$\overline{B_{g_t}(x_0,r)}\subseteq U$ for every $t \in [0,T]$,
where $B_{g_t}(x_0,r)$ denotes the $g_t$-geodesic ball with center
$x_0\in M$ and radius $r$. Then we have
\begin{equation}\label{T4-1-0}
\left|\nabla^T {\rm Rm}_T(x_0)\right|_{T}^2\le C(n){\rm
e}^{C(n)KT}K^2\left(K_1+r^{-2}\right)\left(1-{\rm
e}^{-C(n)(K_1+r^{-2})T}\right)^{-1}
\end{equation}
for some positive constant $C(n)$.
\end{theorem}
\begin{proof}
 Let $X_t^T$, $U_t^T$ be the Brownian motion and
stochastic horizontal lift on $M$ with respect to the time changing
metric $\{\tilde g_t=g_{T-t}\}_{t \in [0,T]}$ whose initial points
are $x_0$ and $u_0 \in (\O, \tilde g_0)$ respectively. Let $Q_t^T$,
$\hat Q_t^T$ be defined by (\ref{Q1}) and (\ref{Q2}) respectively
with $X_t^T$, $U_t^T$ and the operators $F_t$, $\hat F_t$ obtained
above.  By assumption, we can get $|F_t(x)|_t\le CK$, $|\hat
F_t(x)|_t \le CK$ \footnote{For any ${\bf T_t}\in \Gamma({\rm
Hom}(T^{m,k}M,T^{m,k}M))$ and constant $C$, we write $|{\bf
T_t}(x)|_t\le C$ for $x\in M$ if $|T_t\theta(x)|_t\leq
C|\theta(x)|_t$ holds for every $\theta \in T_x^{m,k}M$. We omit the
subscript $0$ when $t=0$.} when $x \in U$, then it is not difficult
to show that
\begin{equation}\label{T4-1-2}
\begin{split}
& |Q_t^T|\le {\rm e}^{CKt},\ |(Q_t^T)^{-1}|\le {\rm e}^{CKt},\ \
|\hat Q_t^T|\le {\rm e}^{CKt},\ \ t \in [0, \tau_U\wedge T],
\end{split}
\end{equation}
where $\tau_U:=\inf\{t>0:\ X_t^T \notin U\}$ is the first exit time
of $U$ for $X_t^T$.

Let $\eta(s):=C{\rm e}^{-\frac{1}{|s-1|^2}}1_{\{|s|<1\}}$ with $C$
being a normalizing constant such that $\int_{\R} \eta(s) \vd s=1$.
We define a function $\bar f_1: \R \rightarrow \R_+$ as follows
\begin{equation*}
\bar f_1(s):=
\begin{cases}
1,\ \ \ \ \ \ \ \ \ \ \ & {\rm if}\ s\le \frac{r}{3},\\
 \frac{r^2-9(s-\frac{r}{3})^2}{r^2},\ \ &{\rm if}\ \frac{r}{3}<s<\frac{2r}{3},\\
 0,\ \ \ \ \ \ \ \ \ \ \ & {\rm if}\ s\ge \frac{2r}{3},
\end{cases}
\end{equation*}
and define $\bar f(s):=\frac{3}{r}\int_{\R} \bar
f_1(s-u)\eta(\frac{3u}{r})\vd u$. It is easy to check that $\bar f
\in C_b^{\infty}(\R),\ \bar f(0)=1,\ \bar f(s)=0$ for every $s \ge
r$, and for every $s \in \R$,
\begin{equation}\label{T4-1-1}
\begin{split}
&  0\le \bar f(s) \le 1,\ -\frac{Cs}{r^2}\le\bar f'(s)\le0 ,\ |\bar
f'(s)|\le \frac{C}{r},\ |\bar f''(s)|\le \frac{C}{r^2}.
\end{split}
\end{equation}
We set $f(t,x):=\bar f(\rho(T-t,x))$ for every $x \in M$, where
$\rho(t,x):=\rho(t,x_0,x)$ is the $g_{t}$-Riemannian distance
between $x_0$ and $x$. Let
\begin{equation*}
\begin{split}
&\Lambda(t):=\int_0^t f^{-2}(s,X_s^T)1_{\{s<\tau_U \wedge T\}}\vd s, \\
&\tau(t):=\inf\{s \ge 0:\ \Lambda(s)\ge t\}\wedge T.
\end{split}
\end{equation*}
By the definition of $\bar f$, we know that $f^{-2}\ge 1$.
Therefore, $\Lambda(t)\ge t$, $\tau(t)\le t$ for every $t \in [0,T]$
and $\Lambda(t) \wedge T=\Lambda(\tau(T))=T$ for every $\tau(T)\le t
\le T$. Moreover, since $\bar f(s)=0$ for every $s \ge r$, by
definition it is easy to check that $\Lambda(t)=\infty$ if there
exists a $s \in (0,t)$ such that $X_s^T \notin
\overline{B_{g_{T-s}}(x_0,r)} $. Thus we have $X_t^T \in
\overline{B_{g_{T-t}}(x_0,r)} \subseteq U$ as long as $t<\tau(T)$.

Let $\hat h \in C^1([0,T]; T^{1,4}\R^n)$ be non-random which will be
determined later such that $\hat h(0)=v$ and $\hat h(T)=0$. Let
$\tilde{h}_t:=\hat h \big(\Lambda(t) \wedge T\big)$. Then
$\tilde{h}_{t}=\hat h(T)=0$ for every $\tau(T)\le t \le T$. Note
that $X_t^T \in U$ as long as $t<\tau(T)$, so we can take $h$ and
$\tau$ to be $\tilde h$ and $\tau(T)$ in Theorem \ref{c1}
respectively (i.e., $h=\tilde{h}$ and $\tau=\tau(T)$) such that
$l^{\tilde h}$ can be defined as in (\ref{c1-1}). Now, we need to
check the condition (\ref{c1-3}). In particular, it suffices to show
\begin{equation*}
\sup_{t \in [0,T]}\E\big[|l^{\tilde h}_t|^2\big]\le
\frac{1}{2}\E\Big[\int_0^T |((Q_t^T)^{\mathrm{tr}})^{-1}(\hat
Q_t^T)^{\mathrm{tr}} \dot{\tilde h}_t|^2 \vd t\Big]<\infty.
\end{equation*}
We have
\begin{equation}\label{T4-1-4}
\begin{split}
& \int_0^T |((Q_t^T)^{\mathrm{tr}})^{-1}(\hat Q_t^T)^{\mathrm{tr}}
\dot{\tilde h}_t|^2 \vd t=\int_0^{\tau(T)}
\left|((Q_t^T)^{\mathrm{tr}})^{-1}(\hat Q_t^T)^{\mathrm{tr}} \dot{\tilde h}_t\right|^2 \vd t\\
&\le \int_0^{\tau(T)}{\rm e}^{CKt} |\dot{\hat
h}\big(\Lambda(t)\big)|^2 f^{-4}(t,X_t^T)\vd t = \int_0^{T}{\rm
e}^{C K\tau(t)} \big|\dot{\hat h}\big(\Lambda(\tau(t))\big)\big|^2
f^{-4}(\tau(t),X_{\tau(t)}^T)\vd (\tau(t))\\
&=\int_0^T {\rm e}^{CK\tau(t)} |\dot{\hat h}(t)|^2
f^{-2}(\tau(t),X_{\tau(t)}^T)\vd t.
\end{split}
\end{equation}
On the other hand, by \cite[Theorem 2]{KP}, for every $t \in [0,T]$,
we have
\begin{equation}\label{T4-1-5}
\begin{split}
 \rho(T-t,X_t^T)=\sqrt{2}B_t+\int_0^t
\left[\Delta^{(T-s)}\rho(T-s,X_s^T)- \frac{\partial
 \rho(T-s,y)}{\partial s}\Big|_{y=X_s^T}\right]1_{\{X_s^T \notin {\{x_0\}\cup\rm
Cut}(T-s,x_0)\}}\vd s -L_t,
\end{split}
\end{equation}
where $B_t$ is a one-dimensional Brownian motion, ${\rm Cut}(t,x_0)$
is the cut-locus of $x_0$ associated with $g_{t}$, and $L_t$ is a
continuous non-decreasing process which only increases when $X_t^T
\in {\rm Cut}(T-t,x_0)$.

For every positive integer $k$, let $\sigma(k):=\inf\{t>0:\
f^{-2}(t,X_t^T)\ge k\}$ and $\tau_k(t):=\tau(t)\wedge \sigma(k)$.
Note that $f(t,X_t^T)=\bar f(\rho(T-t,X_t^T))$. By the It\^{o}
formula and (\ref{T4-1-5}), we have
\begin{equation}\label{T4-1-6}
\begin{split}
&f^{-2}(\tau_k(t),X_{\tau_k(t)}^T)=M_{\tau_k(t)}-
2\int_0^{\tau_k(t)}\bar f'\bar f^{-3}(\rho(T-s,X_s^T))
\mathscr{L}_s \rho(T-s,X_s^T)1_{\{X_s^T \notin \{x_0\}\cup{\rm Cut}(T-s,x_0)\}}\vd s\\
& -2\int_0^{\tau_k(t)}\big(\bar f'' \bar f^{-3}-3(\bar f')^2\bar
f^{-4}\big)(\rho(T-s,X_s^T))\vd s+
\int_0^{\tau_k(t)}2\bar f' \bar f^{-3}(\rho(T-s,X_s^T))\vd L_s\\
&\le M_{\tau_k(t)} -2\int_0^{\tau(t)} \bar f'\bar
f^{-3}(\rho(T-s,X_{s}^T)) \mathscr{L}_{s}
\rho(T-s,X_{s}^T)1_{\{X_{s}^T \notin \{x_0\}\cup{\rm Cut}(T-s,x_0),
s<\sigma(k)\}}
\vd s\\
&-2\int_0^{\tau(t)}\big(\bar f'' \bar f^{-3}-3(\bar f')^2\bar
f^{-4}\big)(\rho(T-s,X_{s}^T))1_{\{s<\sigma(k)\}}\vd s\\
&= -2\int_0^{t} \bar f'\bar f^{-1}(\rho(T-\tau(s),X_{\tau(s)}^T))
\mathscr{L}_{\tau(s)}
\rho(T-\tau(s),X_{\tau(s)}^T)1_{\{X_{\tau(s)}^T \notin
\{x_0\}\cup{\rm Cut}(T-\tau(s),x_0), \tau(s)<\sigma(k)\}}
\vd s\\
&+M_{\tau_k(t)}-2\int_0^{t}\big(\bar f'' \bar f^{-1}-3(\bar
f')^2\bar{f}^{-2}\big)(\rho(T-\tau(s),X_{\tau(s)}^T))1_{\{\tau(s)<\sigma(k)\}}\vd
s \end{split}
\end{equation}
where $M_t:=-2\sqrt{2}\int_0^t \bar f'\bar
f^{-3}(\rho(T-s,X_s^T))\vd B_s$ is a local martingale, and
$\mathscr{L}_t:=\Delta^{(T-t)}-\frac{\partial}{\partial t}$. The
first inequality above is due to the property $\bar f' \bar f^{-3}
\le 0$ and the last step is obtained by the change of variable for
the time parameter.

According to the Laplacian comparison theorem, for every $x \in
\overline{B_{g_{T-t}}(x_0,r)}\subseteq U$ such that $x \notin
\{x_0\}\cup{\rm Cut}(T-t,x_0)$, we can obtain
\begin{equation*}
\Delta^{(T-t)}\rho(T-t,x)\le \sqrt{K_1(n-1)}{\rm
coth}\Big(\sqrt{\frac{K_1}{n-1}}\rho(T-t,x)\Big).
\end{equation*}
Combining this inequality with (\ref{T4-1-1}) and applying the
inequality ${\rm coth}(s) \le 1+s^{-1}$, we get that for every $x
\in \overline{B_{g_{T-t}}(x_0,r)}\subseteq U$ such that $x \notin
\{x_0\}\cup {\rm Cut}(T-t,x_0)$,
\begin{equation*}
-\bar f'(\rho(T-t,x))\bar f
(\rho(T-t,x))\Delta^{(T-t)}\rho(T-t,x)\le
C\left(\frac{\sqrt{K_1}}{r}+\frac{1}{r^2}\right).
\end{equation*}
By \cite[Lemma 5 and Remark 6]{MT} (or \cite[Lemma 4]{KP}), for
every $x \in \overline{B_{g_{T-t}}(x_0,r)}\subseteq U$ such that $x
\notin \{x_0\}\cup {\rm Cut}(T-t,x_0)$,
\begin{equation*}
\begin{split}
& \left|\frac{\partial \rho(T-t,x)}{\partial
t}\right|=\frac{1}{2}\left| \int_0^{\rho(T-t,x)}\frac{\partial
g_{T-t}}{\partial t}(\dot{\gamma}(s),
\dot{\gamma}(s)) \vd s\right|\\
&=\left|\int_0^{\rho(T-t,x)}{\rm Ric}_{T-t}(\dot{\gamma}(s),
\dot{\gamma}(s)) \vd s\right|\le K_1r,
\end{split}
\end{equation*}
where $\gamma(\cdot):[0,\rho(T-t,x)]\rightarrow M$ is the unique
minimizing unit-speed geodesic associated with $g_{T-t}$ connecting
$x$ to $x_0$. Hence by (\ref{T4-1-1}) we can obtain
\begin{equation*}
\left|\bar f'(\rho(T-t,x))\bar f (\rho(T-t,x))\frac{\partial
\rho(T-t,x)}{\partial t}\right|\le CK_1.
\end{equation*}
Due to (\ref{T4-1-1}), we have
\begin{equation*}
\left|\bar f (\rho(T-t,x))\bar
f''(\rho(T-t,x))-3(\bar{f}'(\rho(T-t,x)))^2\right|\le \frac{C}{r^2}.
\end{equation*}
Set $Y^k_t:=f^{-2}(\tau_k(t),X_{\tau_k(t)}^T)$. It is easy to see
that for every $k$,  $\E [M_{\tau_k(t)}]=0$. Note that $X_t^T \in
\overline{B_{g_{T-t}}(x_0,r)}$ as long as $t<\tau(T)$, then based on
(\ref{T4-1-6}) and all the estimates above, it follows that
\begin{equation*}
\begin{split}
&\E [Y_{t}^k]\le 1+C\left(K_1+\frac{1}{r^2}\right)\int_0^t \E
[Y_{s}^k] \vd s.
\end{split}
\end{equation*}
Therefore, using the Grownwall lemma and letting $k \rightarrow
\infty$, we have
\begin{equation}\label{T4-1-7}
\E\big[f^{-2}(\tau(t)\wedge \sigma,X_{\tau(t)\wedge
\sigma}^T)\big]=\lim_{k \rightarrow \infty} \E [Y_t^k] \le {\rm
e}^{C\l(K_1+r^{-2}\r)t},
\end{equation}
where $\sigma:=\inf\{t>0:\ f^{-2}(t,X_t^T)=\infty\}$. Therefore by
(\ref{T4-1-7}) we know $\E\big[f^{-2}(\tau(t)\wedge
\sigma,X_{\tau(t)\wedge \sigma}^T)\big]<\infty$, which implies that
$\tau(t)<\sigma$ a.s. and
\begin{equation}\label{T4-1-8}
\E\big[f^{-2}(\tau(t),X_{\tau(t)}^T)\big]\le {\rm
e}^{C\l(K_1+r^{-2}\r)t}.
\end{equation}

Setting $C(K_1,r):=C(K_1+r^{-2})$ and $\hat
h(t):=\frac{C(K_1,r)v}{1-{\rm e}^{-C(K_1,r)T}}\int_0^{T-t} {\rm e}
^{-C(K_1,r)(T-s)}\vd s$, and putting the above estimate
(\ref{T4-1-8}) into (\ref{T4-1-4}) yields
\begin{equation*}
\begin{split}
&\sup_{t \in [0,T]}\E\big[|l^{\tilde h}_t|^2\big]\le \int_0^T {\rm
e}^{C K t} |\dot{\hat h}(t)|^2
\E\big[f^{-2}(\tau(t),X_{\tau(t)}^T)\big]\vd t\\
&\le {\rm e}^{C K T}\left(\frac{C(K_1,r)|v|}{1-{\rm
e}^{-C(K_1,r)T}}\right)^2 \int_0^T {\rm e}^{-C(K_1,r)t} \vd t = {\rm
e}^{C K T}\frac{C(K_1,r)|v|^2}{1-{\rm e}^{-C(K_1,r)T}}.
\end{split}
\end{equation*}
Then substituting this estimate into (\ref{c1-3}), and noting that
$|N_{t}|\le K$ for every $0\le t \le \tau(T)$, we can derive
(\ref{T4-1-0}) directly.
\end{proof}

\begin{remark}
The only point we need $(M,g_t)$ to be complete in the proof is the
application of Laplacian comparison theorem to the distance function
induced by $g_t$. If for some non-complete manifold, the Laplacian
comparison theorem is still valid (see the example in \cite{Shi}),
then we can also derive the gradient estimate \eqref{T4-1-0}.
\end{remark}

\begin{remark}\label{r1}
{\rm The gradient estimate obtained by Shi \cite{Shi} is as follows
(see also \cite[Theorem 6.15]{BLN}).}

\vskip 3mm

 Suppose that for some constant $K>0$, $|{\rm Rm}_t(x)|_t
\le K$, for every $(t,x) \in [0,T]\times U$, and there exist $x_0
\in M$, $r>0$ such that $\overline{B_{g_0}(x_0,r)}\subseteq U$, then
we have
\begin{equation}\label{r1-1}
\left|\nabla^T {\rm Rm}_T(x_0)\right|_{T}^2\le C(n)
K^2\left(K+r^{-2}+T^{-1}\right)
\end{equation}
for some positive constant $C(n)$.

\vskip 3mm

{\rm Firstly, note that the condition here is
$\overline{B_{g_0}(x_0,r)}\subseteq U$, which seems slightly
different from that in Theorem \ref{T4-1}, which is
$\overline{B_{g_t}(x_0,r)}\subseteq U$ for every $t \in [0,T]$. But
we want to remark that after changing the radius $r$ properly, the
two conditions are equivalent. Clearly, if
$\overline{B_{g_t}(x_0,r)}\subseteq U$ for every $t \in [0,T]$ and
some $r>0$, then $\overline{B_{g_0}(x_0,r)}\subseteq U$. On the
other hand, suppose $\overline{B_{g_0}(x_0,r)}\subseteq U$ for some
$r>0$, by the assumption $|{\rm Rm}_t(x)|_t \le K$ for every $(t,x)
\in [0,T]\times U$, as the same argument in \cite[Lemma 6.10]{BLN},
we have
\begin{equation}\label{r1-1a}
{\rm e}^{-2KT}g_0(x)\le g_t(x) \le {\rm e}^{2KT}g_0(x),\ \ \mbox{for
all}\ \ (t,x)\in [0,T]\times U,
\end{equation}
which implies that $\overline{B_{g_t}\big(x_0,{\rm e}^{-2KT}r\big)}
\subseteq \overline{B_{g_0}(x_0,r)}\subseteq U$ for every $t \in
[0,T]$.

Secondly, as explained in the remark of the proof of  \cite[Theorem
6.15]{Chow}, the fact that $K$ appears in \eqref{r1-1} allows us to
assume without loss of generality in the proof of \eqref{r1-1} that
$T\leq 1/K$. In particular, the case where $T\geq 1/K$ follows from
applying the theorem to the time-translated solution
$\tilde{g}_t:=g_{(T-1/K+t)}$, which satisfies the same curvature
bounds on $[0,1/K]$ as in the original solution on $[0,T-1/K]$ and
is chosen so that $\tilde{g}_{1/K}=g_T$. Now,  we  show that the
estimate (\ref{r1-1}) can be derived from the estimate
(\ref{T4-1-0}) directly. In fact, if $T<(K+r^{-2})^{-1}$, then
$\left(K_1+r^{-2}\right)\left(1-{\rm
e}^{-C(n)(K_1+r^{-2})T}\right)^{-1}\le C_1(n)T^{-1}$ for some
constant $C_1(n)>0$, hence
\begin{equation}\label{r1-2}
\left|\nabla^T {\rm Rm}_T(x_0)\right|_{T}^2\le C(n) K^2T^{-1},\ \ \
\ T<(K+r^{-2})^{-1}.
\end{equation}
If $T\geq (K+r^{-2})^{-1}$, then $\left(1-{\rm
e}^{-C(n)(K_1+r^{-2})T}\right)^{-1}\le C_2(n)$ for some constant
$C_2(n)>0$, hence
\begin{equation}\label{r1-3}
\left|\nabla^T {\rm Rm}_T(x_0)\right|_{T}^2\le C(n)
K^2(K_1+r^{-2}),\ \ \ \ T\geq (K+r^{-2})^{-1}.
\end{equation}
Clearly, combining (\ref{r1-2}) and (\ref{r1-3}), we obtain
(\ref{r1-1}) directly. Moreover, we obtain the following
conclusion.}

\vskip 3mm

Suppose all the assumptions in Theorem \ref{T4-1} hold, then for
every $0<T\le (K+r^{-2})^{-1}$, we have
\begin{equation}\label{r1-4}
\left|\nabla^T {\rm Rm}_T(x_0)\right|_{T}\le C(K,n)T^{-\frac{1}{2}}
\end{equation}
for some constant $C(K,n)>0$.

\end{remark}

Based on the estimate in Theorem \ref{T4-1}, the result on the
higher order gradient estimate can be shown as follows.

\begin{theorem}[\cite{Shi}]\label{T4-1a}
Suppose all the assumptions in Theorem \ref{T4-1} hold. For every
positive integer $m$ and $0<T<(K+r^{-2})^{-1}\wedge 1$, we have
\begin{equation}\label{T4-1a-0a}
\left|(\nabla^T)^m {\rm Rm}_T(x_0)\right|_{T}\le
C(K,n,m)T^{-\frac{m}{2}}
\end{equation}
for some positive constant $C(K,n,m)$.
\end{theorem}
\begin{proof}
The case of $m=1$ is just (\ref{r1-4}). We only prove the case of
$m=2$, and the other cases can be shown similarly and inductively.

According to \cite[Theorem 13.2]{Ham},
\begin{equation}\label{T4-1a-0}
\begin{split}
& \frac{\partial }{\partial t}\nabla^t{\rm Rm}_t= \Delta^t \nabla^t
{\rm Rm}_t+ {\rm Rm}_t*\nabla^t{\rm Rm}_t,\\
&  \frac{\partial }{\partial t}(\nabla^t)^2{\rm Rm}_t= \Delta^t
(\nabla^t)^2 {\rm Rm}_t+ {\rm Rm}_t*(\nabla^t)^2{\rm
Rm}_t+\nabla^t{\rm Rm}_t*\nabla^t{\rm Rm}_t.
\end{split}
\end{equation}
Again for a fixed $T>0$, we set $a_t:=\nabla^{T-t}{\rm Rm}_{T-t}$.
From (\ref{T4-1a-0}) we obtain,
\begin{equation}\label{T4-1a-1}
\begin{split}
& \frac{\partial}{\partial t}a_{t}=-\Delta^{(T-t)} a_{t}+ F_t a_t,\\
& \frac{\partial}{\partial t}\nabla^{(T-t)}
a_t=-\Delta^{(T-t)}(\nabla^{(T-t)} a_t)+ \hat F_t \nabla^{(T-t)}
a_t+\hat G_t,
\end{split}
\end{equation}
where $F_t=-{\rm Rm}_{T-t}*$, $\hat F_t=-{\rm Rm}_{T-t}*$ and $\hat
G_t=-\nabla^{(T-t)}{\rm Rm}_{T-t}*\nabla^{(T-t)}{\rm Rm}_{T-t}$.
Note that comparing with the equation (\ref{e3-1}), we need to estimate
an extra term $\hat G_t$ appearing here by using (\ref{T4-1-0}).

Due to (\ref{r1-1a}), there exists an open domain $\tilde U$ and a
small constant $0<\delta<1$ (which only depends on $K,n$ since we
assume $T<1$)  such that for every $t \in [0,T]$,
\begin{equation}\label{T4-1a-2}
\overline{B_{g_t}(x_0, \delta r)} \subseteq \tilde U\subseteq
\overline{B_{g_t}\l(x_0,\frac{r}{2}\r)} \subseteq
\overline{B_{g_t}(x_0,r)}\subseteq U.
\end{equation}

Let $X_t^T$ and $U_t^T$ be the Brownian motion and stochastic
horizontal lift associated with the time-changing metric $\{\tilde
g_t=g_{T-t}\}_{t \in [0,T]}$. We define $N_t$, $\hat N_t$, $\tilde
h_t$, $l^{\tilde h}_t$, $\Lambda(t)$ and $\tau(t)$ by almost the
same procedure as that in the proof of Theorem \ref{T4-1}, the only
difference is that here we take $v \in T^{1,5}_{x_0}M$,
$a_t=\nabla^{T-t}{\rm Rm}_{T-t}$ and $F_t$, $\hat F_t$ to be the
operators in (\ref{T4-1a-1}), we replace the stopping time $\tau_U$
and the radius $r>0$ by  $\tau_{\tilde U}$ and $\delta r$
respectively, and we set $\tilde h_t:=\hat h\big(\Lambda(t)\wedge
\frac{T}{2} \big)$ for $\hat h(t):= \Big(C(K_1,\delta r)\big(1-{\rm
e}^{-\frac{C(K_1,\delta r)T}{2}}\big)^{-1}\int_0^{\frac{T}{2}-t}
{\rm e} ^{-C(K_1,\delta r)(\frac{T}{2}-s)}\vd s\Big)v$. Therefore,
$\tilde h_t=0$ for every $\tau(\frac{T}{2})\le t \le T$, and $X_t^T
\in \overline{B_{g_t}(x_0, \delta r)}$ as long as
$t<\tau(\frac{T}{2})$.

Let $Z_t:=\langle \hat N_t, \tilde h_t\rangle- \langle  N_t,
l^{\tilde h}_t\rangle$. By carefully tracking the proofs of Theorems
\ref{T1} and \ref{c1}, together with (\ref{T4-1a-1}) we have
\begin{equation}\label{T4-1a-3}
\begin{split}
& \langle (\nabla^T)^2 {\rm Rm}_T)(x_0), v\rangle=\E\big[\langle \hat N_0, \tilde h_0\rangle\big]\\
&= -\E\Big[\langle N_{T \wedge \tau(\frac{T}{2})}, l_{T \wedge
\tau(\frac{T}{2})}^{\tilde h}\rangle\Big] -\E\Big[\int_0^{T \wedge
\tau(\frac{T}{2})}\big\langle \hat Q^T_s \|^{-1}_s\hat G_s(X_s^T),
\tilde h_s\big\rangle \vd s\Big].
\end{split}
\end{equation}

Since for every $y \in \tilde U$, by (\ref{T4-1a-2}) we know that
$\overline{B_{g_t}(y,\frac{r}{2})}\subseteq U$ for every $t \in
[0,T]$.  Therefore, according to Theorem \ref{T4-1} and Remark
\ref{r1},  for every $0<t<(K+r^{-2})^{-1}$ and $y \in \tilde U$, we
have
\begin{equation*}
|\nabla^t {\rm Rm}(y)|_t\le C(K,n)t^{-\frac{1}{2}}.
\end{equation*}
Note that $\tau(\frac{T}{2})\le \frac{T}{2}$ and $X_t^T \in \tilde
U$ for every $t<\tau(\frac{T}{2})$, so for every $T<(K+r^{-2})^{-1}$
and $t<\tau(\frac{T}{2})$,
\begin{equation*}
|\nabla^{T-t}{\rm Rm}_{T-t}(X_t^T)|_{T-t}\le C(K,n)T^{-\frac{1}{2}},
\end{equation*}
which implies that for every $T<(K+r^{-2})^{-1}$,
\begin{equation*}
\E\big[| N_{T \wedge \tau(\frac{T}{2})}|^2\big]\le C(K,n)T^{-1}, \ \
\E\Big[\int_0^{T \wedge \tau(\frac{T}{2})}|\hat{G}_s|\vd s\Big]\le
C(K,n).
\end{equation*}

On the other hand, as explained in the proof of Theorem \ref{T4-1}
and Remark \ref{r1}, if $T<(K+r^{-2})^{-1}$, then
\begin{equation*}
\E\big[|l_{T \wedge \tau(\frac{T}{2})}^{\tilde h}|^2\big]\le
C(K,n)T^{-1},
\end{equation*}
in particular, we use the property $\left(K_1+(\delta
r)^{-2}\right)\left(1-{\rm e}^{-C(n)(K_1+(\delta
r)^{-2})T}\right)^{-1}\le C_1(K,n)T^{-1}$ for every
$T<(K+r^{-2})^{-1}$ since $\delta$ only depends on $K,n$.

Combining all the above estimates together into (\ref{T4-1a-3})
yields the conclusion for $m=2$.
\end{proof}

\begin{remark}\label{r2}
By carefully tracking the proof of Theorems \ref{T4-1} and
\ref{T4-1a}, if we replace the condition that
$\overline{B_{g_t}(x_0,r)}\subseteq U$ for every $t \in [0,T]$ by
the following assumption,
\begin{equation}\label{r2-1}
M(t,r):=\big\{x \in M:\ \eta(t,x)\le r^2 \big\}\subseteq U,\ \forall
t \in [0,T],
\end{equation}
where $\eta:\R_+\times M \rightarrow \R_+$ is a non-negative
$C^{1,2}$ function, such that
\begin{equation*}
\big|\partial_t \eta(t,x)-\Delta^t \eta(t,x)\big|\le C(n),\ \
|\nabla^t \eta(t,x)|_t^2\le C(n)\eta(t,x),\ \ \forall \ (t,x)\in
[0,T]\times U
\end{equation*}
for some constant $C(n)>0$, then  for every positive integer $m$ and
every $0<T<(1+r^{-2})^{-1}$, the gradient estimate \eqref{T4-1a-0a}
is still true.

In fact, under the assumption \eqref{r2-1}, we only need to replace
the function $f(t,x):=\bar f(\rho(T-t,x))$ in the proof of Theorem
\ref{T4-1} by the new function $\tilde f(t,x):=\bar
f(\sqrt{\eta(T-t,x)})$. Then all the arguments in the proof are
still valid due to the application of \eqref{r2-1} to estimate
$\tilde f^{-2}(t,X_t^T)$ in \eqref{T4-1-6}. Thus we can obtain the
estimate \eqref{T4-1a-0a}.

In  particular, the condition \eqref{r2-1} was used in \cite{EH} to
study local gradient estimate for the second fundamental form
evolving by the mean curvature flow.
\end{remark}

\subsection{Local gradient estimate for the second fundamental form of the mean curvature flow}

Throughout this subsection, let $M$ be an $n$-dimensional manifold
and $\phi: M \times [0,T_c)\rightarrow \R^{n+1}$ be a family of
smooth embedding maps with the images $M_t:=\{\phi(x,t):\ x \in
M\}\subseteq \R^{n+1}$ ($n \ge 2$). We suppose that $\phi$ satisfies
the following equation
\begin{equation}\label{e4-2-1}
\begin{cases}
&\frac{\partial}{\partial t}\phi(x,t)=-H(x,t)\nu(x,t),\\
& \phi(\cdot,0)=\phi_0,
\end{cases}
\end{equation}
where $H(x,t)$, $\nu(x,t)$ are the mean curvature and the outward
unit normal vector of hypersurface  $M_t$ at $\phi(x,t)$
respectively, and $\phi_0$ is the initial hypersurface for the flow.


We call $\{\phi(\cdot,t)\}_{t \in [0,T_c)}$ the mean curvature flow
(we write MCF for short) with initial point $\phi_0$. The (smooth)
MCF  was first investigated by G. Huisken in \cite{Hus}, where the
local existence for (\ref{e4-2-1}) was shown when $M$ is compact. For
an overall  introduction to the theory, we refer
readers to \cite{M,Zhux}.

For every fixed $t\in [0,T_c)$, let $\phi(t):=\phi(\cdot,t):M
\rightarrow \R^{n+1}$  be the embedding map and let $\vd \phi(t)^* :
T^*\R^{n+1} \rightarrow  T^*M$ be the pull back map via $\phi(t)$.
Let $\{g_t\}_{t \in [0,T_c)}$ be the metric induced by the MCF,
i.e., $g_t:=\vd \phi(t)^* \bar g$ where $\bar g$ is the standard
Euclidean norm on $\R^{n+1}$. From (\ref{e4-2-1}), it is easy to
check that
\begin{equation*}
\frac{\partial }{\partial t}g_t(x)=-2H(x,t) A(x,t),
\end{equation*}
where the mean curvature $H(x,t)$ is the trace of $A(x,t)$, and
$A(x,t)=\{h_{ij}(x,t)\}$ is the second fundamental form associated
with $\phi(t)$, that is,
\begin{eqnarray*}
 h_{ij}(x,t):= -\left\langle \frac{\partial^2
\phi(x,t)}{\partial x_i \partial x_j }, \nu(x,t)\right \rangle,
\end{eqnarray*}
where, in particular,  $\langle\cdot, \cdot \rangle$ denotes  the
Euclidean inner product on $\R^{n+1}$.

In this subsection, we assume that the equation (\ref{e4-2-1}) holds
on $U \times [0,T]$ for some constant $0<T<T_c$ and open domain $U
\subseteq M$, $(M,g_t)$ is complete for every $t \in [0,T_c)$, but
with no requirement that $M$ is compact. We will apply Theorem
\ref{c1} to give a local gradient estimate for $A$. Note that by
\cite[Proposition 2.3.1]{M} and \cite[Lemma 2.3.4]{M}, for each
$(x,t) \in U \times [0,T]$, we have
\begin{equation}\label{e4-2-2}
\begin{split}
& \frac{\partial }{\partial t}A(t)=\Delta^t A(t)+A(t)*A(t)*A(t),\\
& \frac{\partial }{\partial t}\nabla^t A(t)= \Delta^t \nabla^t
A(t)+A(t)*A(t)*\nabla^t A(t).
\end{split}
\end{equation}
So, from (\ref{e4-2-2}), we know that $a_t:=A(T-t)$ satisfies the
equation (\ref{e3-1}) with respect to the time changing metric
$\{\tilde g_t:=g_{T-t}\}_{t \in [0,T]}$ and the operators $F_t$ and
$\hat F_t$ should be chosen to be $F_t=-A(T-t)*A(T-t)*$ and $\hat
F_t=-A(T-t)*A(T-t)*$, respectively.

We can obtain the following estimate which was first shown by K.
Ecker and G. Huisken \cite[Theorem 3.7]{EH} by using some analytic
method.

\begin{theorem}[\cite{EH}]\label{T4-2}
Suppose that for some constant $K>0$, $|A(x,t)|_t \le K$ for every
$(x,t) \in U\times [0,T]$, and there exist $x_0 \in M$, $r>0$ such
that $\overline{B_{g_t}(x_0,r)}\subseteq U$ for every $t \in [0,T]$,
then under the MCF \eqref{e4-2-1} we have
\begin{equation}\label{T4-2-1}
|\nabla^T A(x_0,T)|_{T}^2\le C(n){\rm
e}^{C(n)K^2T}K^2\left(K^2+r^{-2}\right)\left(1-{\rm
e}^{-C(n)(K^2+r^{-2})T}\right)^{-1},
\end{equation}
for some positive constant $C(n)$. Moreover, for every positive
integer $m$ and $0<T<(K^2+r^{-2})^{-1}\wedge 1$, we have
\begin{equation}\label{T4-2-2}
\left|(\nabla^T)^m A(x_0,T)\right|_{T}\le C(K,n,m)T^{-\frac{m}{2}}
\end{equation}
for some positive constant $C(K,n,m)$.
\end{theorem}
\begin{proof}
As in the proof of Theorem \ref{T4-1}, let $X_t^T$, $U_t^T$ be the
Brownian motion and stochastic horizontal lift associated with the
time changing metric $\{\tilde g_t\}$ respectively, and let $Q_t^T$,
$\hat Q_t^T$ be separately defined by (\ref{Q1}) and (\ref{Q2}) with
respect to $X_t^T$, $U_t^T$ and the operators $F_t$, $\hat F_t$
given above.

Note that $|F_t(x)|_t\le CK^2$, $|\hat F_t(x)|_t \le CK^2$, and
$|\partial_t g_t(x)|_t \le CK^2$ when $(x,t) \in U \times [0,T]$.
Moreover, by the Gauss equation, we have $|{\rm Ric}_t(x)|\le CK^2$
for every $(x,t) \in U \times [0,T]$. Then we can repeat the
procedure in the proof of Theorem \ref{T4-1} to prove
(\ref{T4-2-1}).

Also note that by \cite[Lemma 2.3.4]{M}, for every non-negative
integer $m$,
\begin{equation*}
\begin{split}
& \frac{\partial }{\partial t}(\nabla^t)^m A(t)= \Delta^t
(\nabla^t)^m A(t)+\sum_{i+j+k=m,\ i,j,k \in \mathbb{N}}(\nabla^t)^i
A(t)*(\nabla^t)^j A(t)*(\nabla^t)^k A(t).
\end{split}
\end{equation*}
Then following the same procedure as that in the proof Theorem
\ref{T4-1a}, we can show (\ref{T4-2-2}) based on (\ref{T4-2-1}).
\end{proof}
\begin{remark}
As explained in Remark \ref{r2}, we can also obtain the above
estimate under the assumption \eqref{r2-1}, which is exactly the
conclusion of \cite[Theorem 3.7]{EH}.
\end{remark}

\subsection{Local gradient estimate for the second fundamental form of the forced mean curvature flow}

\subsubsection{Type I}

Let $\phi:M\times [0,T_c)\rightarrow \R^{n+1}$ ($n\geq2$) be a
family of smooth embedding maps evolving as follows
 \begin{eqnarray} \label{4.3.1}
 \left\{
 \begin{array}{lll}
 \frac{\partial}{\partial
 t}\phi(x,t)=-H(x,t)\nu(x,t)+\kappa(t)\phi(x,t), \\
 \\
 \phi(\cdot,0)=\phi_0, & \quad
 \end{array}
 \right.
\end{eqnarray}
where, as in Subsection 4.2, $H(x,t)$ and $\nu(x,t)$ are the mean
curvature and outward unit normal vector of the hypersurface
$M_t:=\{\phi(x,t):\ x \in M\}$ at $\phi(x,t)$, $\phi_0$ is the
initial hypersurface, and $\kappa(t)$ is a bounded continuous
function on $[0,T_c)$. Clearly, the flow (\ref{4.3.1}) can be
obtained by adding a forcing term $\kappa(t)\phi(t,x)$ to the
classical MCF in the direction of the position vector. In fact, the
forced flow (\ref{4.3.1}) can be seen as an extension of the MCF,
since it degenerates to  the MCF if $\kappa(t)\equiv0$. As pointed
out in \cite{lmw}, this forced MCF (\ref{4.3.1}) is different from
the flow (\ref{e4-2-1}), this is because for the forcing term
$\kappa(t)\phi(x,t)$ in (\ref{4.3.1}), although the tangent
component of $\phi(x,t)$ does not affect the behavior of the
evolving hypersurfaces, the normal component of $\phi(x,t)$ is
usually not $\nu (x,t)$. More precisely, the normal component
$\langle\phi(x,t),v(x,t)\rangle$ is not only depending on the
time-variable $t$ but also depending on the space-variable
$x\in{M}$. Readers can find that the convergence situation of the
flow (\ref{4.3.1}) is more complicated than that of the MCF even if
the initial hypersurface is a sphere (see \cite[Remark 2.2]{m3}).
This flow was firstly introduced by Mao, Li and Wu \cite{mlw} for
considering an entire graph evolving under this forced flow, and
later they also investigated an $n$-dimensional ($n\geq2$) compact
and strictly convex hypersurface evolving under the flow
(\ref{4.3.1}) in \cite{lmw}. In these two cases, by making
discussions on the component function $\kappa(t)$ of the forcing
term, some convergence results can be obtained (cf. \cite[Main
Theorem]{mlw}, \cite[Theorem 1.1]{lmw}). Furthermore, Mao \cite{m2}
studied the evolution of two-dimensional graphs in $\mathbb{R}^4$
under the flow (\ref{4.3.1}). Clearly, this is a high-codimensional
MCF problem, and naturally, in general, doing estimates for the
geometric quantities, like the second fundamental form, the mean
curvature, etc, in the case of high-codimension is more difficult
than the case of hypersurface, since in the case of
high-codimension, the evolving submanifolds have at least two normal
vectors which leads to the complexity of evolution equations of
those geometric quantities. One thing being worthy to be pointed out
here is that Mao \cite{m1} has improved this spirit of adding a
forcing term in direction of the position vector to the case of
hyperbolic MCF, and some convergence results have been shown therein.

\emph{To avoid confusion, in the rest part of this subsection we
will follow the usage of notations in Subsection 4.2}. That is,
$g_t$ is the induced metric on the evolving hypersurface $\phi(t)$,
and $h_{ij}(t)$ denotes the component of the second fundamental form
$A(t)$. From (\ref{4.3.1}), it is easy to get
\begin{eqnarray*}
\frac{\partial }{\partial t}g_t=-2H(t) A(t)+2\kappa(t)g_{t},
\end{eqnarray*}
which has an extra term $2\kappa(t)g_{t}$ compared with the case of
classical MCF. We assume now that the equation (\ref{4.3.1}) holds
on $U \times [0,T]$ for some constant $0<T<T_c$ and open domain $U
\subseteq M$, $(M,g_t)$ is complete for every $t \in [0,T_c)$, but
with no requirement that $M$ is compact. By \cite[Lemma 2.2]{lmw}
and \cite[Lemma 6.12]{lmw}, for each $(x,t) \in U \times [0,T]$, we
have
\begin{equation}\label{4.3.2}
\begin{split}
& \frac{\partial }{\partial t}A(t)=\Delta^t
A(t)+A(t)*A(t)*A(t)+\kappa(t)A(t), \\
& \frac{\partial }{\partial t}\nabla^t A(t)= \Delta^t \nabla^t
A(t)+A(t)*A(t)*\nabla^t A(t)+\kappa(t)\nabla^t A(t).
\end{split}
\end{equation}
So, from (\ref{4.3.2}), we know that $a_t:=A(T-t)$ satisfies the
equation (\ref{e3-1}) with respect to the time changing metric
$\{\tilde g_t:=g_{T-t}\}_{t \in [0,T]}$ and the operators $F_t$ and
$\hat F_t$ should be chosen as
$F_t=\big(-A(T-t)*A(T-t)*\big)-\kappa(T-t)$ and $\hat
F_t=\big(-A(T-t)*A(T-t)*\big)-\kappa(T-t)$, respectively.

\begin{theorem} \label{theoremxx1}
Suppose that for some constant $K>0$, $|A(x,t)|_t \le K$ for every
$(x,t) \in U \times[0,T]$, and there exist $x_0 \in M$, $r>0$ such
that $\overline{B_{g_t}(x_0,r)}\subseteq U$ for every $t \in [0,T]$,
then under the flow (\ref{4.3.1}) we have
\begin{equation}\label{th4-3-1}
\left|\nabla^T A(x_0,T)\right|_{T}^2\le C(n){\rm
e}^{\left[C(n)(K^2+k_{+})\right]T}K^2\left(K^2+k_{+}+r^{-2}\right)\left(1-{\rm
e}^{-C(n)(K^2+k_{+}+r^{-2})T}\right)^{-1}
\end{equation}
for some positive constant $C(n)$, where
$k_{+}:=\sup\limits_{0\leq{t}\leq T}|\kappa(t)|$ is the supremum of
$\kappa(t)$ on $[0,T]$. Moreover, for every positive integer $m$ and
$0<T<(K^2+k_{+}+r^{-2})^{-1}\wedge 1$, we have
\begin{equation}\label{th4-3-2}
\left|(\nabla^T)^m A(x_0,T)\right|_{T}\le
C(K,k_+,n,m)T^{-\frac{m}{2}}
\end{equation}
for some positive constant $C(K,k_+,n,m)$.
\end{theorem}
\begin{proof}
Now, let $X_t^T$, $U_t^T$ be the Brownian motion and stochastic
horizontal lift associated with the time changing metric $\{\tilde
g_t\}$ respectively, and let $Q_t^T$, $\hat Q_t^T$ be separately
defined by (\ref{Q1}) and (\ref{Q2}) with respect to $X_t^T$,
$U_t^T$ and the operators $F_t$, $\hat F_t$ given above.

By assumption, we have $|F_t(x)|_t\le C(K^2+k_{+})$, $|\hat
F_t(x)|_t \le C(K^2 +k_{+})$ and  $|\partial_t g_t(x)|_t \le C(K^2
+k_{+})$ when $(x,t) \in U \times [0,T]$. Moreover, by the Gauss
equation, we have $|{\rm Ric}_t(x)|\le CK^2$ if $(x,t) \in U \times
[0,T]$. Then we can repeat the procedure in the proof of Theorem
\ref{T4-1} to show (\ref{th4-3-1}).

According to \cite[Lemma 6.12]{lmw}, for every non-negative integer
$m$,
\begin{equation*}
\begin{split}
& \frac{\partial }{\partial t}(\nabla^t)^m A(t)= \Delta^t
(\nabla^t)^m A(t)+\sum_{i+j+k=m,\ i,j,k \in \mathbb{N}}(\nabla^t)^i
A(t)*(\nabla^t)^j A(t)*(\nabla^t)^k A(t)+\kappa(t)(\nabla^t)^m A(t).
\end{split}
\end{equation*}
Then using (\ref{th4-3-1}) and following the same procedure as that
in the proof of Theorem \ref{T4-1a}, we can show (\ref{th4-3-2}).
\end{proof}

\subsubsection{Type II}

Let $\phi: M \times [0,T_c) \rightarrow \R^{n+1}$ ($n\geq2$) be a
family of smooth embedding maps evolving as follows
 \begin{eqnarray} \label{4.3.3}
 \left\{
 \begin{array}{lll}
 \frac{\partial}{\partial
 t}\phi(x,t)=\left(\kappa(t)-H(x,t)\right)\nu(x,t), \\
 \\
 \phi(\cdot,0)=\phi_0, & \quad
 \end{array}
 \right.
\end{eqnarray}
where $H(x,t)$, $\nu(x,t)$ and $\phi_0$ have the same meanings as
those in Subsection 4.2, and $\kappa(t)$ is a function with respect
to the time variable $t$.

The forced MCF (\ref{4.3.3}) has been extensively studied by many
mathematicians. For instance, in the setting that $M_0=\phi_0(M)$ is
compact and strictly convex, if $\kappa(t)=\int_{M_t}H\vd
\mu_{t}/\int_{M_t}\vd\mu_{t}$ (i.e., the average of the mean
curvature on $M_{t}=\phi(M,t)$) with $\vd\mu_{t}$ the volume element
of $M_t$, G. Huisken \cite{Hus2} has proved that this is a
volume-preserving flow which exists for all the time, i.e.,
$t\in[0,\infty)$, and $M_{t}$ converges to a round sphere; if
$\kappa(t)=\int_{M_t}H^{2}\vd\mu_t/\int_{M_t}\vd\mu_t$, J. McCoy
\cite{mccoy} has proved that this is a surface-area preserving flow
which also exists for all the time and converges to a round sphere;
furthermore, if
$\kappa(t)=\int_{M_t}HE_{k+1}\vd\mu_t/\int_{M_t}E_{k+1}\vd\mu_t$
with $E_{l}$ the $l$-th elementary symmetric function of the
principle curvatures of $M_t$, McCoy \cite{mccoy2} has investigated
this case and has shown that the flow exists for all the time and
converges to a round sphere which generalizes the results of the
volume-preserving MCF \cite{Hus2} and the area-preserving MCF
\cite{mccoy}. Because of this, McCoy called this flow \emph{the
mixed volume-preserving MCF}. The trivial case is $\kappa(t)=0$, and
in this case the flow (\ref{4.3.3}) degenerates into the classical
MCF, and then G. Huisken's convergence result in \cite{Hus} can be
applied. Clearly, different forcing terms lead to different
existence and convergence results for the forced flow (\ref{4.3.3}).
A natural question can be issued, that is, ``\emph{how to unify all
the cases?}". Li and Salavessa \cite{li} have given a positive
answer to this problem.

From (\ref{4.3.3}), it is easy to get
\begin{eqnarray*}
\frac{\partial }{\partial t}g_t=2\left(\kappa(t)-H(t)\right)A(t).
\end{eqnarray*}
We assume now that the equation (\ref{4.3.3}) holds on $U \times
[0,T]$ for some constant $0<T<T_c$ and open domain $U \subseteq M$,
$(M,g_t)$ is complete for every $t \in [0,T_c)$, but with no
requirement that $M$ is compact. By the argument in \cite{Hus2} (for
instance, \cite[Proposition 1.1]{Hus2} and the last equality on Page
46 of \cite{Hus2}), for each $(x,t) \in U \times [0,T]$, we have
\begin{equation}\label{4.3.4}
\begin{split}
& \frac{\partial }{\partial t}A(t)=\Delta^t
A(t)+A(t)*A(t)*A(t)+\kappa(t)A(t)*A(t), \\
& \frac{\partial }{\partial t}\nabla^t A(t)= \Delta^t \nabla^t
A(t)+A(t)*A(t)*\nabla^t A(t)+\kappa(t)A(t)*\nabla^t A(t).
\end{split}
\end{equation}
Moreover, for every non-negative integer $m$,
\begin{equation*}
\begin{split}
& \frac{\partial }{\partial t}(\nabla^t)^m A(t)= \Delta^t
(\nabla^t)^m A(t)+\sum_{i+j+k=m,\ i,j,k \in \mathbb{N}}(\nabla^t)^i
A(t)*(\nabla^t)^j A(t)*(\nabla^t)^k A(t)\\
&+\sum_{i+j=m,\ i,j \in \mathbb{N}}\kappa(t)(\nabla^t)^i
A(t)*(\nabla^t)^j A(t).
\end{split}
\end{equation*}
So, from (\ref{4.3.4}), we know that $a_t:=A(T-t)$ satisfies the
equation (\ref{e3-1}) with respect to $\{\tilde g_t:=g_{T-t}\}_{t
\in [0,T]}$ and the operators $F_t$ and $\hat F_t$ should be chosen
to be $F_t=\big(-A(T-t)*A(T-t)*\big)-\kappa(T-t)A(T-t)*$ and $\hat
F_t=\big(-A(T-t)*A(T-t)*\big)-\kappa(T-t)A(T-t)*$, respectively.

Now, similar to Theorem \ref{theoremxx1}, we can obtain the
following.

\begin{theorem} \label{theoremxx2}
Suppose that for some constant $K>0$, $|A(x,t)|_t \le K$ for every
$(x,t) \in U\times[0,T]$, and there exist $x_0 \in M$, $r>0$ such
that $\overline{B_{g_t}(x_0,r)}\subseteq U$ for every $t \in [0,T]$,
then under the flow (\ref{4.3.3}) we have
\begin{equation*}
\left|\nabla^T A(x_0,T)\right|^2_{T}\le C(n){\rm
e}^{C(n)\left[K+k_{+}\right]KT}K^2\left(K^2+k_{+}K+r^{-2}\right)\left(1-{\rm
e}^{-C(n)(K^2+k_{+}K+r^{-2})T}\right)^{-1}
\end{equation*}
for some positive constant $C(n)$, where
$k_{+}:=\sup\limits_{0\leq{t}\leq T}|\kappa(t)|$ is the supremum of
$\kappa(t)$ on $[0,T]$. Moreover, for every positive integer $m$ and
$0<T<(K^2+k_{+}K+r^{-2})^{-1}\wedge 1$, we have
\begin{equation*}
\left|(\nabla^T)^m A(x_0,T)\right|_{T}\le
C(K,k_+,n,m)T^{-\frac{m}{2}}
\end{equation*}
for some positive constant $C(K,k_+,n,m)$.
\end{theorem}

\begin{proof}
The proof is almost the same as that of Theorem \ref{theoremxx1}
except that in this case, by assumption, we have $|F_t(x)|_t\le
C(K^2+k_{+}K)$, $|\hat F_t(x)|_t \le C(K^2 +k_{+}K)$, and
$|\partial_t g_t(x)|_t \le C(K^2 +k_{+}K)$ when $(x,t) \in U \times
[0,T]$.
\end{proof}

\subsection{Local gradient estimate for the scalar curvature of the Yamabe flow}

We consider the Yamabe flow defined by the following evolution
equation
\begin{equation}\label{e5-1}
\left\{
  \begin{array}{lll}
    \frac{\partial}{\partial t}g_t(x)=-{\rm R}_t(x) g_t(x), & 0\le t<T_c, \\
    \\
    g_0(x)=\mathring{g}(x)
  \end{array}
\right.
\end{equation}
on an $n$-dimensional ($n\geq 3$) complete Riemannian manifold
$(M,\mathring{g})$, where ${\rm R}_t$ is the $g_t$-scalar curvature
associated with the metric $g_t$, and $\mathring{g}$ is the initial
metric for the flow.

The Yamabe flow was proposed by R. S. Hamilton \cite{Ham89} in the
1980¡¯s as a tool for constructing metrics of constant scalar
curvature in a given conformal class, where  the existence of a
global solution has been shown  for every initial metric. Hamilton's
pioneering work gives another approach to the Yamabe problem solved
by R. Schoen \cite{Sch}.

We assume that the equation \eqref{e5-1} holds on $[0,T]\times U$
for some constant $T\in(0,T_c)$ and open domain $U\subset M$,
$(M,g_t)$ is complete for every $t \in [0,T_c)$, but with no
requirement that $M$ is compact. Note that by
\cite[Lemma 2.2]{Chow}, for each $(t,x)\in [0,T]\times U$, we have
\begin{align}
&\frac{\partial}{\partial t}{\rm R}_t=(n-1)\Delta^t {\rm R}_t+{\rm R}_t^2, \label{e5-2}\\
&\frac{\partial}{\partial t} \nabla^t{\rm
R}_t=(n-1)\Delta^t(\nabla^t{\rm R}_t)+3{\rm R}_t\nabla^t {\rm R}_t.
\label{e5-3}
\end{align}
Set $a_t:={\rm R}_{\frac{T-t}{n-1}}$, from \eqref{e5-2} and
\eqref{e5-3} we can easily check that $a_t$ satisfies (\ref{e3-1})
associated with the metric $\{\tilde g_t=g_{\frac{T-t}{n-1}}\}$ and
the operators $F_t:=-\frac{a_t}{n-1}$ ,
$\hat{F}_t:=-\frac{3a_t}{n-1}$.

By applying Theorem \ref{c1},
we can give a local gradient estimate for the Yamabe flow (see
\cite[Theorem 3.4]{CZ12}).

\begin{theorem}[\cite{CZ12}] \label{theoremxx3}
Suppose that for some constant $K>0$, $|{\rm Ric}_t|\leq K$ for
every $(t,x)\in [0,T]\times U$, and there exist $x_0\in M$, $r>0$
such that $\overline{B_{g_t}(x_0,r)}\subset U$ for every $t\in
[0,T]$. Then, we have
\begin{align}\label{4.3.5}
|\nabla^T{\rm R}_T(x_0)|_T^2\leq \frac{C(n){\rm
e}^{C(n)KT}K^2(K+r^{-2})}{1-{\rm e}^{-C(n)(K+r^{-2})T}}
\end{align}
for some positive constant $C(n)$.
\end{theorem}
\begin{proof}
As in the proof of Theorem \ref{T4-1a}, let $X_t^T$, $U_t^T$ be the
Brownian motion and stochastic horizontal lift associated with the
time changing metric $\{\tilde g_t\}$ respectively, and let $Q_t^T$,
$\hat Q_t^T$ be separately defined by (\ref{Q1}) and (\ref{Q2}) with
respect to $X_t^T$, $U_t^T$ and the operators $F_t$, $\hat F_t$
given above.

Note that, $|F_t(x)|_t\leq CK$,  $|\hat{F}_t(x)|\leq CK$, and
$|\partial_t \tilde g_t|_t\le CK$ for every $(t,x)\in [0,T]\times
U$, and $\overline{B_{\tilde g_t}(x_0, r)}\subset U$ for every $t\in
[0,T]$.
Then, following the same steps in the proofs of Theorems \ref{T4-1}
and \ref{T4-1a} we can prove the conclusion.
\end{proof}
\begin{remark}\label{r2-2}
By a similar discussion as in Remark \ref{r1}, from the estimate
\eqref{4.3.5} we can obtain
$$|\nabla^T{\rm R}_T(x_0)|_T^2\leq C(n)K^2\l(K+\frac{1}{r^2}+\frac{1}{T}\r)$$
for some positive constant $C(n)$. This form is consistent with
\cite[Theorem 3.4]{CZ12}.
\end{remark}

\section*{Acknowledgments}
\renewcommand{\thesection}{\arabic{section}}
\renewcommand{\theequation}{\thesection.\arabic{equation}}
\setcounter{equation}{0} \setcounter{maintheorem}{0}

 The
second author was supported by the starting up research fund (Grant
No. 109007329) supplied by Zhejiang University of Technology. The
third author was partially supported by the starting-up research
fund (Grant No. HIT(WH)201320) supplied by Harbin Institute of
Technology (Weihai), the project (Grant No. HIT.NSRIF.2015101)
supported by Natural Scientific Research Innovation Foundation in
Harbin Institute of Technology,  and the NSF of China (Grant No.
11401131). The last revision of this paper before submission was
carried out when the third author visited the Chern Institute of
Mathematics (CIM), Nankai University in December 2013, and the third
author here is grateful to Prof. Shao-Qiang Deng for the hospitality
during his visit to CIM. The third author would like to thank CIM
for supplying the financial support during his visit through the
Visiting Scholar Program. The authors would like to thank the
anonymous referee for his or her careful reading and valuable
comments such that the article appears as its present version.

 \end{document}